\documentclass[journal]{IEEEtran}  % Comment this line out if you need a4paper

\pdfoutput=1
\usepackage[utf8]{inputenc}
\usepackage[pdftex]{graphicx}
\graphicspath{{./Figures/}{../Figures/}} % My standard location for storing figures
\usepackage{epstopdf} % For using eps files in figures
\usepackage{float} % For being able to make environments float
\usepackage{caption}
\usepackage{subcaption} % For subfigures
\usepackage{booktabs}
\usepackage{amsmath, amssymb} % For various mathematical commands and symbols
\usepackage{amsthm} % For theorems
\usepackage{multirow}
\usepackage{cases}
\usepackage{tocloft} % For the table of contents
\usepackage{enumitem} % For enumerates

\usepackage{threeparttable}
\usepackage{tablefootnote}

\usepackage{amsthm} % For theorems, assumptions, remarks, etc
    \newtheorem{lemma}{Lemma}
    
    \newtheorem{definition}{Definition}
    
    \newtheorem{remark}{Remark}
    \newtheorem{theorem}{Theorem}

\renewcommand{\qedsymbol}{$\blacksquare$}

\usepackage{dblfloatfix}

\usepackage[ruled, lined, longend, linesnumbered]{algorithm2e}
\usepackage[hang,flushmargin]{footmisc}
\usepackage{algpseudocode}
\SetKwInput{Input}{Input}
\SetKwInput{Declares}{Declares}
\SetKwInput{Require}{Require}
\SetKwInput{Prototype}{Prototype}
\SetKw{KwBy}{by}
\SetKw{Kwor}{or}
\SetKw{Kwand}{and}
\SetKw{KwEndFor}{end for}

\usepackage{fancyhdr}

\def\Ind#1{\cc{I}_{#1}} % Indicator function
\def\set#1#2{\{  #1 : #2 \}}
\def\R{\mathbb{R}} % Set of real numbers
\def\Z{\mathbb{Z}} % Set of integer numbers
\def\Dp#1{\mathbb{D}_{++}^{#1}} % Set of diagonal matrices of dimension #1
\def\Sp#1{\mathbb{S}_{++}^{#1}} % Set of positive definite matrices of dimension #1
\newcommand{\cc}[1]{{\mathcal{#1}}} % Short for giving it the font for sets

\newcommand{\T}{^\top} % Transpose
\def\cP{\cc{P}} % Euclidean projection onto a set
\def\vecZeros{\mathbf{0}} % For a vector (or matrix) of zeros
\def\sp#1#2{\langle #1,#2\rangle } % Scalar product
\def\Sum#1#2{\sum\limits_{#1}^{#2}} % Sum with limits
\def\bmat#1{\left[\begin{array}{#1}} % Begin matrix
\def\emat{\end{array}\right]} % End matrix

\def\fracg#1#2{{\displaystyle{\frac{#1}{#2}}}} % Big \frac
\def\vv#1{{ \rm \bf{#1}}} % Bold
\newcommand{\becauseof}[2][=]{\stackrel{\scriptstyle\mkern-1.5mu#2\mkern-1.5mu}{#1}}
\def\moveEq#1{{}\mkern#1mu} % Move equation to left or right

\def\cK{\cc{K}} % Set of the second order cone constraint
\def\cKl{\cK_{-}}
\def\cKu{\cK_{+}}
\def\cD{\cc{D}} % Set ued for the union of the two SOC constraints.
\newcommand{\yLB}{\underline{y}} \newcommand{\yLBj}{\underline{y}_{(i)}}
\newcommand{\yUB}{\overline{y}}
\newcommand{\yUBj}{\overline{y}_{(i)}}

\def\xe{x_e} % Notation for x_e
\def\xs{x_s}
\def\xc{x_c}
\def\ue{u_e} % Notation for u_e
\def\us{u_s}
\def\uc{u_c}
\def\ye{y_e}
\def\yej{y_{e (i)}}
\def\ys{y_s}
\def\ysj{y_{s (i)}}
\def\yc{y_c}
\def\ycj{y_{c (i)}}

\def\xh{x_h} % Notation for the sequence of x_{h,j}
\def\xhj{x_{h}^{j}} % Notation for each element of the sequence x_{h,j}

\def\xH{\vv{x}_H} % Notation for the shorthand for (x_e, x_s, x_c)
\def\uh{u_h} % Notation for the sequence of u_{h,j}
\def\uhj{u_{h}^{j}} % Notation for each element of the sequence u_{h,j}

\def\uH{\vv{u}_H} % Notation for the shorthand for (u_e, u_si, u_ci)

\def\bz{\tilde{z}}
\newcommand{\zUB}{\overline{z}}
\newcommand{\zLB}{\underline{z}}

\begin{document}
% Fakesection Title
\pagestyle{fancy}

\title{Efficiently solving the harmonic model predictive control formulation}

\author{Pablo~Krupa$^\dagger$,~Daniel~Limon$^\dagger$,~Alberto~Bemporad$^*$,~Teodoro~Alamo$^\dagger$%
\thanks{$^\dagger$ Department of Systems Engineering and Automation, Universidad de Sevilla, 41092, Sevilla, Spain. E-mails: \texttt{pkrupa@us.es}, \texttt{dlm@us.es}, \texttt{talamo@us.es}. Corresponding author: Pablo Krupa.}%
\thanks{$^*$ IMT School for Advanced Studies, Piazza San Francesco 19, Lucca, Italy. Email: \texttt{alberto.bemporad@imtlucca.it}}
\thanks{This work was supported
    in part by Grant PDC2021-121120-C21 funded by MCIN/AEI/10.13039/501100011033 and by the ``European Union NextGenerationEU/PRTR",
    and in part by Grant Margarita Salas (grant number 20122) funded by the Ministerio de Universidades and the European Union (NextGenerationEU).}%
}

\maketitle
\thispagestyle{fancy}

% Fakesection abstract
\begin{abstract}
    Harmonic model predictive control (HMPC) is a model predictive control (MPC) formulation which displays several benefits over other MPC formulations, especially when using a small prediction horizon.
    These benefits, however, come at the expense of an optimization problem that is no longer the typical quadratic programming problem derived from most linear MPC formulations due to the inclusion of a particular class of second order cone constraints.
    This article presents a method for efficiently dealing with these constraints in operator splitting methods, leading to a computation time for solving HMPC in line with state of the art solvers for linear MPC.
    We show how to apply this result to the alternating direction method of multipliers algorithm, presenting a solver which we compare against other solvers from the literature, including solvers for other linear MPC formulations.
    The results show that the proposed solver, and by extension the HMPC formulation, is suitable for its implementation in embedded systems.
\end{abstract}

\begin{IEEEkeywords}

predictive control, harmonic model predictive control, convex optimization, embedded systems, ADMM.
\vspace*{-0.5em}

\end{IEEEkeywords}

\section{Introduction}

In the recent publication \cite{Krupa_TAC_2022} (originally presented in \cite{Krupa_CDC_19}), the authors proposed a novel model predictive control (MPC) \cite{Camacho_S_2013} formulation labeled \textit{harmonic model predictive control} (HMPC), which has several advantages over other MPC formulations, such as guaranteed asymptotic stability, recursive feasibility even in the event of a sudden reference change, an increased domain of attraction with respect to other MPC formulations, does not require a positive invariant set of the system, and displays an improved performance of the closed-loop system when using a small prediction horizon, especially for systems with integrator states or slew-rate constraints.

These  advantages, which are highlighted and discussed in detail in \cite{Krupa_TAC_2022, Krupa_CDC_19}, indicate that HMPC is an ideal candidate for its use as an embedded controller, i.e., for its implementation in devices with low computation and memory resources.
The main drawback, however, is that its optimization problem is not the typical quadratic programming (QP) problem derived from most linear MPC formulations.
This is due to the inclusion of several constraints that can be imposed as second-order cone (SOC) constraints.

This article shows that in spite of this drawback, the HMPC formulation can be solved in a computation time comparable to that of MPC formulations whose control law is derived from QP problems using state of the art solvers.

We show how to efficiently deal with the SOC-like constraints of the HMPC formulation by grouping them in pairs and considering their intersection.
We prove an explicit solution to the Euclidean projection onto this intersection of pairs of SOC-like constraints; a fact that can be exploited by several first-order methods.
In particular, the operator splitting methods considered in solvers such as \cite{Garstka_JOTA_2021}, \cite{Stellato_OSQP} or \cite{ODonoghue_SCS_21} can make good use of this, since the pairing (along with the explicit solution of the projection operator) leads to a reduction of the number of decision variables and dimensions of the matrices involved in the solver, leading to a reduction of the computation time when compared to simply solving the original problem by considering SOC constraints.

To show how to solve the resulting optimization problem, we present a solver based on the alternating direction method of multipliers (ADMM) algorithm \cite{Boyd_FTML_2011} using ideas and approaches taken from state of the art solvers \cite{Garstka_JOTA_2021}, \cite{Stellato_OSQP}, \cite{Krupa_TCST_20}.
The resulting solver, which is available in the SPCIES toolbox \cite{SPCIES}, is well suited for its implementation in embedded systems, especially considering that the HMPC formulation is particularly suited for its use with small prediction horizons.

The remainder of this article is structured as follows.
In Section \ref{sec:SOC} we describe the SOC-like sets that we deal with in subsequent sections and present explicit solutions for the Euclidean projection onto them.
We briefly recall the HMPC formulation in Section \ref{sec:HMPC}.
Section \ref{sec:solver} presents the ADMM algorithm and its particularization to HMPC.
Numerical results are presented in Section \ref{sec:results}.
Concluding remarks and a discussion of the computational results are provided in Section~\ref{sec:conclusions}.

% Fakesection Notation
\subsubsection*{Notation}

Given two integers $i$ and $j$ with ${j \geq i}$, $\Z_i^j$ denotes the set of integer numbers from $i$ to $j$, i.e. ${\Z_i^j \doteq \{i, i+1, \dots, j-1, j\}}$. 
We denote by $\Sp{n}$ ($\Dp{n}$) the set of (diagonal) positive definite matrices in $\R^{n \times n}$.
Given a set $\cc{X} \subseteq \R^n$, we denote by $\Ind{\cc{X}}$ its indicator function, i.e., $\Ind{\cc{X}}(x) = 0$ if $x \in \cc{X}$ and $\Ind{\cc{X}}(x) = +\infty$ if $x \not\in \cc{X}$.
For vectors $x_1$ to $x_N$, $(x_{1}, x_{2}, \dots, x_{N})$ denotes the column vector formed by their concatenation.
Given a vector $x\in \R^{n}$, we denote its $i$-th component using a parenthesized subindex $x_{(i)}$.
Given two vectors $x \in \R^{n}$ and $y \in \R^{n}$, their standard inner product is denoted by $\sp{x}{y} \doteq \sum_{i=1}^{n} x_{(i)} y_{(i)}$.
For $x \in \R^{n}$ and $A \in \Sp{n}$, $\|x\| \doteq \sqrt{\sp{x}{x}}$, $\|x\|_A \doteq \sqrt{\sp{x}{A x}}$, $\| x \|_\infty \doteq \max_{i = 1 \dots n}{| x_{(i)} |}$.
The Euclidean projection of a vector $x \in \R^n$ onto a set $\cc{X} \subseteq \R^n$ is denoted by $\cP_\cc{X}(x)$, i.e., $\cP_\cc{X}(x) = \arg\min_{v \in \cc{X}} \| v - x \|^2$.

\section{Shifted second order cones} \label{sec:SOC}

This section describes a class of closed convex sets, which we denote by \textit{shifted second order cones} (shifted-SOCs).
We prove an explicit solution for the Euclidean projection onto them and onto the intersection of two ``opposed" shifted-SOCs.
The results and definitions of this section will play a major role in subsequent developments.

\begin{definition}[Shifted second order cone] \label{def:K}
A \textit{shifted second order cone} (shifted-SOC) $\cK_\alpha(c) \subset \R^n$ is a set given by
\begin{equation} \label{eq:K}
    \cK_\alpha(c) = \set{z = (z_0, z_1) \in \R \times \R^{n-1}}{\|z_1\| \leq \alpha(z_0 - c)},
\end{equation}
where $\alpha \in \{ 1, -1\}$ and $c \in \R$.
For convenience, let us denote by $\cKl(c) \doteq \cK_{-1}(c)$ and $\cKu(c) \doteq \cK_{1}(c)$, where we may drop the ``$(c)$" if it is clear from the context.
\end{definition}

The following theorem provides an explicit solution for the Euclidean projection onto $\cK_\alpha(c)$.
Its proof is heavily inspired by the proof of \cite[Theorem 3.3.6]{Bauschke_Thesis_1996}, which proves an explicit solution onto \eqref{eq:K} for $c = 0$ and $\alpha > 0$.

\begin{theorem} \label{theo:proj:K}
    Let $z = (z_0, z_1) \in \R \times \R^{n-1}$ and $\cK_{\alpha}(c) \subset \R^n$ be given by Definition \ref{def:K} for some $\alpha \in \{1, -1\}$ and $c \in \R$.
    The Euclidean projection of $z$ onto $\cK \doteq \cK_\alpha(c)$ is given by
    \begin{subnumcases}{\cc{P}_{\cK}(z) = } \label{eq:projection:K}
    z & $\text{if } \|z_1\| \leq \alpha ( z_0 - c )$ \label{eq:proj:K:1} \\
    (c, 0) & $\text{if } \|z_1\| \leq -\alpha ( z_0 - c )$ \label{eq:proj:K:2}\\
    \left( \tau \alpha + c, \fracg{\tau z_1}{\|z_1\|} \right) & $\text{otherwise,}$ \label{eq:proj:K:3}
\end{subnumcases}
where $\tau = \fracg{1}{2}(\alpha (z_0 - c) + \|z_1\|)$. 
\end{theorem}

\begin{proof}\renewcommand{\qedsymbol}{}
    See Appendix \ref{app:proof:proj:K}.
\end{proof}

We are now interested in the following set, obtained from the intersection of two ``opposed" shifted-SOCs.
Let us denote by $\cD(\zUB, \zLB) \subset \R^n$, where $\zUB, \zLB \in \R$, the set given by
\begin{equation} \label{eq:D}
    \cD(\zUB, \zLB) \doteq \set{z \in \R^n}{z \in \cKl(\zUB) \cap \cKu(\zLB)}.
\end{equation}
Once again, we may drop the ``$(\zUB, \zLB)$" if it is clear from the context.
Set $\cD$ is closed and convex, since it is the intersection of two closed convex sets \cite[Prop. 1.1.1(a)]{Bertsekas_Convex_2009}, \cite[Prop. A.2.4(b)]{Bertsekas_Convex_2009}.
Additionally, it is non-empty if $\zLB \leq \zUB$, as we state in the following lemma.

\begin{lemma} \label{lemma:nonempty:D}
Let $\cD(\zUB, \zLB) \subset \R^n$ be given by \eqref{eq:D} for some $\zUB, \zLB \in \R$. Then, $\cD$ is non-empty iff $\zLB \leq \zUB$.
\end{lemma}

\begin{proof}
First, assume that $\cD$ is non-empty and take any $z\in\cD$.
Then, from $z\in\cKl(\zUB)$ and $z\in\cKu(\zLB)$ we have that $\|z_1\| \leq \zUB - z_0$ and $\|z_1\| \leq z_0 - \zLB$, which leads to $\zUB - \zLB \geq 2 \| z_1 \| \geq 0$.
Next, assume that $\zLB \leq \zUB$ and consider the vector $z = ((\zUB +\zLB)/2, 0) \in \R\times\R^{n-1}$. It is easy to verify that $z\in\cKl$ and $z\in\cKu$, thus $z\in\cD$.
\end{proof}

The projection onto set $\cD$ could be performed using one of many methods from the literature for projecting onto the intersections of convex sets \cite{Boyle_1986}, \cite{Stovsic_projection_2016}.
These methods typically consider the Euclidean projection of a vector onto a non-empty closed convex set $\cc{C} = \cc{C}_1 \cap \cc{C}_2 \cap \cdots \cap \cc{C}_r$, where $r > 0$ is finite, and it is assumed that the sets $\cc{C}_i$ are closed and convex and that $\cP_{\cc{C}_i}$ has a known solution for $i \in \Z_1^r$.
They are employed because, in general, the projection onto the intersection of convex sets is not guaranteed to be the result of projecting onto each set $\cc{C}_i$ in order, even if $r = 2$.

However, we will show that this is not the case for the projection onto set $\cD$, which can be obtained by first projecting onto $\cKu$ and then projecting the resulting vector onto $\cKl$, both of which have simple explicit solutions given by Theorem \ref{theo:proj:K}.
This result will allow us to directly use sets $\cD$ in future developments without having to result to an iterative method to compute the projection onto them, which would be computationally expensive, especially if a good approximation of the projection is required.

\begin{algorithm}[t] % Dykstra's algorithm
    \DontPrintSemicolon
    \SetAlgoNoEnd
    \caption{Dykstra's algorithm for $\cc{C} = \cc{C}_1 \cap \cc{C}_2$.} \label{alg:Dykstra}
    \KwIn{$z \in \R^n$}
    $w^0 \gets z$, $p^0 \gets \vecZeros_n$, $q^0 \gets \vecZeros_n$\;
    \ForEach{$k \geq 1$}{
        $v^k \gets \cP_{\cc{C}_1} (w^{k-1} + p^{k-1})$\;
        $p^k \gets w^{k-1} + p^{k-1} - v^k$\;
        $w^k \gets \cP_{\cc{C}_2} (v^k + q^{k-1})$\;
        $q^k \gets v^k + q^{k-1} - w^k$\;
    }
\end{algorithm}

The following theorem states a direct solution of the projection onto a non-empty set $\cD$.
Its proof is based on making use of the following lemma, which states the condition for the convergence after a single iteration of Algorithm \ref{alg:Dykstra}, obtained from \cite[\S 3]{Bauschke_Dykstra_94}, which is a particularization of Dykstra's algorithm \cite{Boyle_1986} to finding the projection of $z \in \R^n$ onto $\cc{C} = \cc{C}_1 \cap \cc{C}_2$.
It generates iterates $v^k$ and $w^k$ satisfying $\|v^k - \cP_{\cc{C}}(z) \| \rightarrow 0$ and $\| w^k - \cP_\cc{C}(z) \| \rightarrow 0$ as $k \rightarrow + \infty$ \cite[Theorem 2]{Boyle_1986}.
% In the following lemma we show the condition for the convergence of Algorithm \ref{alg:Dykstra} after a single iteration.

\begin{lemma} \label{lemma:convergence:Dykstra}
Let $\cc{C} = \cc{C}_1 \cap \cc{C}_2$ be a non-empty closed convex set and $\cc{C}_1, \cc{C}_2 \subseteq \R^n$ be closed convex sets.
Consider Algorithm \ref{alg:Dykstra} for finding $\cP_\cc{C}(z)$ for $z \in \R^n$.
Then, $v^2 = w^1 \Rightarrow w^1 = \cP_\cc{C}(z)$.
\end{lemma}

\begin{proof}
The first iterate of Algorithm \ref{alg:Dykstra} satisfies
\begin{equation*}
    v^1 = \cP_{\cc{C}_1}(z), \quad p^1 = z - v^1, \quad w^1 = \cP_{\cc{C}_2}(v^1), \quad q^1 = v^1 - w^1.
\end{equation*}
Then, if $v^2 = w^1$, we have that
\begin{align*}
    &v^2 = \cP_{\cc{C}_1}(w^1 + p^1) = w^1, \;\; p^2 = w^1 + p^1 - w^1 = p^1, \\
    &w^2 = \cP_{\cc{C}_2}(v^2 + v^1 - w^1) = \cP_{\cc{C}_2}(v^1) = w^1, \\
    &q^2 = v^2 + q^1 - w^2 = q^1.
\end{align*}
Therefore, iterations $k > 2$ will return the same results as $k = 2$, meaning we have reached a fixed point of the algorithm, and thus $w^2 = w^1 = \cP_{\cc{C}}(z)$ \cite[Theorem 2]{Boyle_1986}.
% Since this algorithm converges to the projection of $z$ onto $\cc{C}$ \cite[Theorem 2]{Boyle_1986}, i.e., that $\| w^k - \cP_\cc{C}(z) \| \rightarrow 0$ as $k \rightarrow +\infty$
\end{proof}

\begin{theorem} \label{theo:proj:D}
    Let ${z = (z_0, z_1) \in \R \times \R^{n-1}}$ and $\cD(\zUB, \zLB)$ be the set given by \eqref{eq:D} for some $\zUB, \zLB \in \R$ satisfying $\zLB \leq \zUB$. Then, the projection of $z$ onto $\cD$ is given by $\cP_\cD(z) = \cP_{\cKl} \left( \cP_{\cKu} (z) \right)$.
    % \begin{equation} \label{eq:projection:D}
    % \end{equation}
\end{theorem}

\begin{proof}\renewcommand{\qedsymbol}{}
    See Appendix \ref{app:proof:proj:D}.
\end{proof}

% \begin{remark}
% This result is useful because it will allow us to directly work with sets $\cD$ in future developments.
% The result of the theorem is not immediate, since in general the projection onto a non-empty closed convex set $\cc{C} = \cc{C}_1 \cap \cc{C}_2 \cap \cdots \cap \cc{C}_r$, where $r > 0$ is finite and $\cc{C}_i$ for $i \in \Z_1^r$ are closed convex sets, is not guaranteed to be the result of projecting onto each set $\cc{C}_i$ in order, even in the case $r = 2$.
% Assuming that the projection onto each set $\cc{C}_i$ can be performed, Dykstra's algorithm \cite{Boyle_1986} generates a sequence of iterates that converge to the projection onto $\cc{C}$ by iteratively projecting onto the sets $\cc{C}_i$.
% This, however, can be computationally expensive, especially if a good approximation of the projection is required.
% \end{remark}

\section{Harmonic model predictive control} \label{sec:HMPC}

The HMPC formulation \cite{Krupa_TAC_2022}, \cite{Krupa_CDC_19}, considers a controllable linear time-invariant system described by the discrete state space model
\begin{equation} \label{eq:model}
    x(t+1) = A x(t) + B u(t),
\end{equation}
where $x(t) \in \R^{n_x}$ and $u(t) \in \R^{n_u}$ are the state and control input at the discrete time instant $t$, respectively, subject to
\begin{equation} \label{eq:constraints}
    \yLB \leq E x(t) + F u(t) \leq \yUB,
\end{equation}
where we assume that the bounds $\yLB, \yUB \in \R^{n_y}$ satisfy $\yLB < \yUB$.

The HMPC formulation is inspired by the \textit{MPC for tracking} (MPCT) formulation \cite{Ferramosca_A_2009}, \cite{Limon_A_2008}, whose difference with classical MPC formulations is that it includes an \textit{artificial reference}, which is forced to be a steady-state of \eqref{eq:model} satisfying \eqref{eq:constraints},
as decision variables in the optimization problem

The idea behind HMPC is to substitute this steady-state artificial reference by one in the form of the periodic signals $x_h^j$, $u_h^j$, called the \textit{artificial harmonic reference}, whose value at each discrete time instant $j\in\Z$ is given by
\begin{subequations} \label{eq:harmonic:signals}
\begin{align}
    \xhj &= \xe + \xs \sin(w ({j}{-}{N})) + \xc \cos(w ({j}{-}{N})), \label{eq:x_h}\\
    \uhj &= \ue + \us \sin(w ({j}{-}{N})) + \uc \cos(w ({j}{-}{N})), \label{eq:u_h}
\end{align}
\end{subequations}
i.e., to use a harmonic signal, thus the name of the formulation, with base frequency $w \geq 0$ parameterized by $\xe$, $\xs$, $\xc \in \R^{n_x}$ and $\ue$, $\us$, $\uc \in \R^{n_u}$.
Let us introduce the following notation to simplify future developments:
\begin{align*}
    &\xH \doteq (\xe, \xs, \xc) \in \R^{n_x} \times \R^{n_x} \times \R^{n_x}, \\
    &\uH \doteq (\ue, \us, \uc) \in \R^{n_u} \times \R^{n_u} \times \R^{n_u}, 
\end{align*}
\vspace{-2.4em}
\begin{align*}
    \ye = E \xe + F \ue, \; \ys = E \xs + F \us, \; \yc = E \xc + F \uc.
\end{align*}

For a given prediction horizon $N > 0$ and base frequency $w \geq 0$, the HMPC control law for a given state $x(t) \in \R^{n_x}$ and reference $(x_r, u_r) \in \R^{n_x} \times \R^{n_u}$ is derived from
\begin{subequations} \label{eq:HMPC} % HMPC
\begin{align}
    \min\limits_{\substack{\vv{x},\vv{u},\\ \xH, \uH}} \;& \Sum{j = 0}{N-1} \ell_h(x^j, u^j, \xhj, \uhj) + V_h(\xH, \uH; x_r, u_r) \label{eq:HMPC:cost} \\
    {\rm s.t.} \;& x^0 = x(t) \label{eq:HMPC:cond:inic}\\
    &x^{j+1} = A x^j + B u^j, \; j\in\Z_0^{N-1} \label{eq:HMPC:dynamics}\\
    & \yLB \leq E x^j + F u^j \leq \yUB, \; j\in\Z_0^{N-1} \label{eq:HMPC:constraints}\\
    & A x^{N-1} + B u^{N-1} = \xe + \xc \label{eq:HMPC:xN}\\
    & \xe = A \xe + B \ue \label{eq:HMPC:xe}\\
    & \xs \cos(w) - \xc \sin(w) = A \xs + B \us \label{eq:HMPC:xs}\\
    & \xs \sin(w) + \xc \cos(w) = A \xc + B \uc \label{eq:HMPC:xc}\\
    & (\yej, \ysj, \ycj) \in \cD(\yUBj, \yLBj), \; i\in\Z_1^{n_y}, \label{eq:HMPC:D}
\end{align}
\end{subequations}
where $\vv{x} = (x^0, \dots, x^{N-1} )$, $\vv{u} = ( u^0, \dots, u^{N-1} )$, and the two terms of the cost function are given by the stage cost function
\begin{equation*}
    \ell_h(x, u, \xh, \uh) = \| x - \xh \|_Q^2 + \| u - \uh \|_R^2,
\end{equation*}
where $Q\in\Sp{n_x}$ and $R\in\Sp{n_u}$; and the offset cost function 
\begin{align*}
    V_h(\cdot) &= \| \xe - x_r \|_{T_e}^2 + \| \ue - u_r \|_{S_e}^2 \nonumber \\
               &\quad+ \| \xs \|_{T_h}^2 + \| \xc \|_{T_h}^2 + \| \us \|_{S_h}^2 + \| \uc \|_{S_h}^2,
\end{align*}
where $T_e\in\Sp{n_x}$, $T_h \in\Dp{n_x}$, $S_e\in\Sp{n_u}$, and $S_h \in\Dp{n_u}$.
Note that constraint \eqref{eq:HMPC:D} is imposing $n_y$ constraints onto sets $\cD \subset \R^3$ defined in \eqref{eq:D}.

As is typical in MPC, the summation of the stage cost function $\ell_h$ is penalizing the discrepancy between the predicted states $x^j$ and inputs $u^j$ with the reference, although in this case the discrepancy is with respect to the artificial reference at prediction time $j$, i.e., $\xhj$ and $\uhj$, respectively.
The offset cost function is penalizing two conceptually different things:
\begin{itemize}
    \item The distances $\| \xe - x_r \|^2_{T_e}$ and $\| \ue - u_r \|^2_{S_e}$.
        This penalization will make the ``center" of the artificial harmonic reference signal move towards the reference, eventually reaching it if is an admissible steady-state of the system.
    \item The magnitude of $\xs$, $\xc$, $\us$ and $\uc$.
        This will force the artificial harmonic reference signal to converge to the steady-state $(\xe, \ue)$ as $k \rightarrow +\infty$, since it will converge towards $\xs = \xc = 0$ and $\us = \uc = 0$ as $k \rightarrow \infty$.
\end{itemize}
The result of this, as stated and proven in \cite[Theorem 3]{Krupa_TAC_2022}, is that the closed-loop system will asymptotically converge to $(x_r, u_r)$ if it is an admissible steady-state of the system, or to the admissible steady-state $(x_t, u_t)$ that minimizes the distance $\| x_t - x_r \|^2_{T_e} + \| u_t - u_r \|^2_{S_e}$ otherwise.
In both cases, the closed-loop system will satisfy the constraints \eqref{eq:constraints} under nominal conditions.

Constraints \eqref{eq:HMPC:cond:inic}-\eqref{eq:HMPC:constraints} impose the typical MPC constraints: current system state, system dynamics and system constraints, respectively.
Constraint \eqref{eq:HMPC:xN} forces the terminal state $x^N$ to reach the artificial harmonic reference, i.e., $x^N = x_{h}^N$, and \eqref{eq:HMPC:xe} that $(x_e, u_e)$ must be a steady-state of the system.
The satisfaction of \eqref{eq:HMPC:xs}-\eqref{eq:HMPC:xc} along with \eqref{eq:HMPC:xe} results in an artificial harmonic reference \eqref{eq:harmonic:signals} that satisfies the system dynamics \eqref{eq:model}, i.e., it satisfies $x_{h}^{j+1} = A \xhj + B \uhj$, $\forall j$ (see \cite[Property 2]{Krupa_TAC_2022}).
Finally, the constraints \eqref{eq:HMPC:D} enforce that the artificial harmonic reference satisfies the system constraints \eqref{eq:constraints} (see \cite[Property 3]{Krupa_TAC_2022}).
The reason for imposing the system constraints on the artificial harmonic reference this way is that the satisfaction of \eqref{eq:HMPC:D} implies the feasibility of $\xhj$ and $\uhj$ for all $j$.
Therefore, the system constraints \eqref{eq:constraints} can be imposed on the artificial reference by only $n_y$ constraints.
A more naive approach would have resulted in a number of constraints that depends on the values of $N$ and $w$.
The downside, however, is that we add constraints on the intersection of shifted-SOC constraints, leading to an optimization problem that is no longer the typical QP obtained from most linear MPC formulations.

For a more detailed description and discussion of this formulation and its advantages we refer the reader to \cite{Krupa_TAC_2022}, \cite{Krupa_CDC_19}.

\begin{remark}
The HMPC formulation \eqref{eq:HMPC} can be posed as a SOC programming problem by first separating each constraint \eqref{eq:HMPC:D} into two constraints, one on $\cKu$ and one on $\cKl$.
In this case, problem \eqref{eq:HMPC} can be solved using SOC programming solvers such as \cite{ODonoghue_SCS_21}.
However, by doing so, we have $2 n_y$ SOC constraints instead of the $n_y$ constraints \eqref{eq:HMPC:D}.
Our use of the sets $\cD$ leads to a smaller number of constraints, which can have a major impact on the computation time of the solver when $n_y$ is large or when using certain operator splitting approaches such as the one used in \cite{Garstka_JOTA_2021}, \cite{Stellato_OSQP}, \cite{ODonoghue_SCS_21}.
\end{remark}

\section{ADMM solver for HMPC} \label{sec:solver}

This section shows how the HMPC formulation can be solved taking into account the sets and projection theorems from Section \ref{sec:SOC} by presenting a solver based on the ADMM algorithm \cite{Boyd_FTML_2011}.

\subsection{Alternating direction method of multipliers} \label{sec:solver:SADMM}

Let us consider the optimization problem
\begin{subequations} \label{eq:general:OP}
\begin{align}
    \min\limits_{z, s} \;& \frac{1}{2} z\T H z + q\T z \\
    {\rm s.t.} \;& G z = b, \; s \in \cc{S} \label{eq:general:OP:constraints} \\
           & C z + s = d, \label{eq:general:OP:relation}
\end{align}
\end{subequations}
where $z \in \R^{n}$ are the \textit{primal decision variables}, $s \in \R^{m}$ are the \textit{primal slack variables}, $H \in \Sp{n}$, $q \in \R^{n}$, $G \in \R^{n_{eq} \times n}$, $b \in \R^{n_{eq}}$, $C \in \R^{m \times n}$, $d \in \R^{m}$ and $\cc{S}$ is a Cartesian product of the form
$\cc{S} = \cc{S}_1 \times \cc{S}_2 \times \cdots \times \cc{S}_{p}$,
where $\cc{S}_i \subseteq \R^{m_i}$, $i \in \Z_1^{p}$, with $\sum_{i = 1}^{p} m_i = m$, are non-empty closed convex sets, i.e., we consider a partition of $s$ given by $s = (s_1, s_2, \dots, s_{p})$, where each $s_i \in \R^{m_i}$, $i \in \Z_1^{p}$, must belong to $\cc{S}_i$.

Problem \eqref{eq:general:OP} can be solved by considering \eqref{eq:general:OP:relation} as the linear constraint relating the two separable decision variables $z$ and $s$, and including \eqref{eq:general:OP:constraints} as indicator functions in the objective function.
That is, to consider the classical ADMM optimization problem
\begin{align*}
    \min\limits_{z, s} \;& f(z) + g(s) \\
    {\rm s.t.} \;& C z + s = d,
\end{align*}
where $f(z) = \frac{1}{2} z\T H z + q\T z + \Ind{(G, b)}(z)$, $g(s) = \Ind{\cc{S}}(s)$, and $\Ind{(G, b)}$ is the indicator functions of set $\set{z \in \R^{n}}{G z = b}$.
This leads to the Lagrangian function
\begin{equation*}
    \cc{L}(z, s, \lambda) = f(z) + g(s) + \frac{\rho}{2} \| C z + s - d + \frac{1}{\rho} \lambda \|^2,
\end{equation*}
where $\lambda \in \R^{m}$ are the dual variables for constraint \eqref{eq:general:OP:relation} and $\rho > 0$ is the penalty parameter.

Starting at an initial point $(z^0, s^0, \lambda^0)$, the iterations of ADMM consist of the following steps \cite{Boyd_FTML_2011}:
\begin{subequations}
\begin{align}
    z^{k+1} &= \arg\min\limits_{z} \cc{L}(z, s^k, \lambda^k) \label{eq:ADMM:step:z} \\
    s^{k+1} &= \arg\min\limits_{s} \cc{L}(z^{k+1}, s, \lambda^{k}) \label{eq:ADMM:step:s} \\
    \lambda^{k+1} &= \lambda^{k} + \rho (C z^{k+1} + s^{k+1} - d), \label{eq:ADMM:step:lambda}
\end{align}
\end{subequations}
where $k \geq 0$ is the iteration counter.
Step \eqref{eq:ADMM:step:z} solves
\begin{equation} \label{eq:ADMM:step:z:QP}
    z^{k+1} = \arg\min\limits_{z} \;\left\{ \frac{1}{2} z\T \hat{H} z + (\hat{q}^k)\T z, \;\;{\rm s.t.} \; G z = b \right\},
\end{equation}
where $\hat{H} = H + \rho C\T C$ and $\hat{q}^k = q + C\T \left( \rho (s^k - d) + \lambda^k \right)$.
Step \eqref{eq:ADMM:step:s} solves
\begin{equation*}
    s^{k+1} = \arg\min\limits_{s \in \cc{S}} \; \frac{\rho}{2} \| C z^{k+1} + s - d + \frac{1}{\rho} \lambda^k \|^2,
\end{equation*}
which is the Euclidean projection of $(-C z^{k+1} + d - \frac{1}{\rho} \lambda^k)$ onto $\cc{S}$, that considering its separability, can be solved for each subset $s_i^{k+1}$, $i \in \Z_1^p$.

\subsection{Particularization to the HMPC's optimization problem} \label{sec:solver:HMPC:poly}

Problem \eqref{eq:HMPC} can be recast as \eqref{eq:general:OP} by taking $s = (s_b, s_c)$
\begin{equation*}
    z = (u^0, x^1, u^1, \dots, x^{N-1}, u^{N-1}, \xe, \xs, \xc, \ue, \us, \uc), \\
\end{equation*}
where $s_b = (s_{b, 0}, s_{b, 1}, \dots s_{b, N-1})$ with $s_{b, i} \in \R^{n_y}$, $i \in \Z_0^{N-1}$, and $s_c = (s_{c, 1}, s_{c, 2}, \dots, s_{c, n_y})$ with $s_{c, i} \in \R^3$, $i \in \Z_1^{n_y}$.
That is, $s_{b, i}$ accounts for the constraints \eqref{eq:HMPC:constraints} and $s_{c, i}$ for the constraints \eqref{eq:HMPC:D}.
Therefore, set $\cc{S}$ in \eqref{eq:general:OP:constraints} is given by
\begin{equation*}
    \cc{S} = \underbrace{\cc{C} \times \cc{C} \times \cdots \times \cc{C}}_{N} \times \cc{D}_1 \times \cc{D}_2 \times \cdots \times \cc{D}_{n_y},
\end{equation*}
where $\cD_i \doteq \cD(\yUBj, \yLBj) \subset \R^3$, $i \in \Z_1^{n_y}$, as defined in \eqref{eq:D}, and $\cc{C} \doteq \set{y \in \R^{n_y}}{\yLB \leq y \leq \yUB}$.
Ingredients $G$ and $b$ account for the equality constraints \eqref{eq:HMPC:cond:inic}, \eqref{eq:HMPC:dynamics}, \eqref{eq:HMPC:xN}, \eqref{eq:HMPC:xe}, \eqref{eq:HMPC:xs} and \eqref{eq:HMPC:xc}; $C$ and $d$ for the relationship between $z$ and $s$; and $H$ and $q$ for the cost function \eqref{eq:HMPC:cost}.

\begin{remark}
We note that $x^0$ does not appear in the above definition of $z$ because equation \eqref{eq:HMPC:cond:inic} can be implicitly considered in the optimization problem, i.e., taking the first equation of \eqref{eq:HMPC:dynamics} as $x^1 = A x(t) + B u^0$, as is typical in MPC solvers.
\end{remark}

By defining $z$ and, in particular, $s$ in this way, we have that the update of $s^{k+1}$ requires two types of projections: the first is a projection onto $\cc{C}$, which is a simple projection onto a box, and the second is a projection onto $\cD_i$ for each $i \in \Z_1^{n_y}$, whose explicit solution is provided in Theorem \ref{theo:proj:D}.

Finally, we must find a way of solving the equality constrained QP \eqref{eq:ADMM:step:z:QP}, corresponding to step \eqref{eq:ADMM:step:z} of the ADMM algorithm.
There are multiple ways to do this \cite{Benzi_2005}.
A popular approach in sparse QP solvers is the one presented in \cite[\S 3.1]{Stellato_OSQP}, in which its KKT conditions are expressed as a linear system of equations whose solution can be sparsely computed using matrix decompositions such as the QR \cite{Saraf_TAC_2019} or LDL$\T$ factorizations \cite[\S 2]{Garstka_JOTA_2021}, \cite[\S 3.1]{Stellato_OSQP}.
Other approaches make use of the Cholesky decomposition \cite{Krupa_TCST_20}, leading to very sparse and simple matrices thanks to the simple structure of $G$ in linear MPC.
However, in the case of HMPC, we find that the straightforward explicit solution from \cite[\S 10.1.1]{Boyd_ConvexOptimization} provides the best computational results.
After some simple algebraic manipulations, this approach leads to
\begin{equation} \label{eq:SADMM:solve:system:dense}
    z^{k+1} =  M_q \hat{q}^k + M_b b,
\end{equation}
where $M_q$ and $M_b$ are the matrices given by
\begin{align*}
    M_q &= \hat{H}^{-1} G\T (G \hat{H}^{-1} G\T )^{-1} G \hat{H}^{-1} - \hat{H}^{-1}, \\
    M_b &= \hat{H}^{-1} G\T (G \hat{H}^{-1} G\T )^{-1}.
\end{align*}
Although matrices $M_q$ and $M_b$ are generally dense, we find this approach is often favorable compared to sparse approaches due to the fact that HMPC is of particular interest when using a small prediction horizon, leading to relatively small-dimensional matrices, where the advantages of the sparse approaches are no longer meaningful.
We also note that only the first $n_x$ columns of $M_b$ are required, since all the elements of $b$ are always zero with the exception of the first $n_x$.

\begin{algorithm}[t]
    \DontPrintSemicolon
    \caption{ADMM for solving HMPC} \label{alg:ADMM:HMPC}
    \Require{$x(k) \in \R^{n_x}$, $(x_r, u_r) \in \R^{n_x} \times \R^{n_u}$, $z^0 \in \R^{n}$, $s^0 \in \R^{m}$, $\lambda^0 \in \R^{m}$, $\epsilon_p > 0$, $\epsilon_d > 0$}
    $k \gets 0$\;
    Update $q$ with $x_r$ and $u_r$ and $b$ with $x(t)$\;
    \Repeat{$ \| c \|_\infty \leq \epsilon_p $ \Kwand $\| s^k - s^{k-1} \|_\infty \leq \epsilon_d$ \label{alg:ADMM:HMPC:exit}}{ 
        $\hat{q}^k \gets q + C\T \left( \rho (s^k - d) + \lambda^k \right)$ \label{alg:ADMM:HMPC:q_hat}\;
        $z^{k+1} \gets$ Solution of \eqref{eq:ADMM:step:z:QP} for $\hat{q}^k$ \label{alg:ADMM:HMPC:z}\;
        $c \gets C z^{k+1} - d$ \label{alg:ADMM:HMPC:c1}\;
        \For{$i \in \Z_0^{N-1}$}{
            $s_{b, i}^{k+1} \gets \max ( \min ( -c_{b, i} - \frac{1}{\rho} \lambda_{b, i}^{k},\; \yUB),\; \yLB)$ \label{alg:ADMM:HMPC:box}\;
        }
        \For{$i \in \Z_1^{n_y}$}{
            $s_{c, i}^{k+1} \gets \cP_{\cKl(\yUBj)} \left( \cP_{\cKu(\yLBj)} ( -c_{c, i} - \frac{1}{\rho} \lambda_{c, i}^{k}) \right)$ \label{alg:ADMM:HMPC:SOC}\;
        }
        $c \gets c + s^{k+1}$ \label{alg:ADMM:HMPC:c2}\;
        $\lambda^{k+1} \gets \lambda^{k} + \rho c$\;
        $k \gets k + 1$\;
    }
    \KwOut{$\tilde{z} \gets z^{k}$, $\tilde{s} \gets s^k$}
\end{algorithm}

Algorithm \ref{alg:ADMM:HMPC} shows the particularization of ADMM to the HMPC formulation \eqref{eq:HMPC}, where we take the same partition of vectors $c$ and $\lambda$ that we took for $s$ (see the beginning of this subsection) in steps \ref{alg:ADMM:HMPC:box} and \ref{alg:ADMM:HMPC:SOC}, i.e., $c = (c_b, c_c)$ and $\lambda = (\lambda_b, \lambda_c)$.
We note that step \ref{alg:ADMM:HMPC:SOC} is making use of Theorem~\ref{theo:proj:K}.
The algorithm returns a suboptimal solution $(\tilde{z}, \tilde{s})$, where the level of suboptimality is determined by the choice of the exit tolerances $\epsilon_p > 0$ and $\epsilon_d > 0$ \cite[\S 3.3]{Boyd_FTML_2011}.

Vector $c$ is used to store the values of $C z^{k+1} - d$ and $C z^{k+1} + s - d$.
We use it to reduce to one the number of times the operation $C z^{k+1}$ is performed in each iteration of the algorithm.
Even though $C$ is sparse and the multiplication operations in steps \ref{alg:ADMM:HMPC:q_hat} and \ref{alg:ADMM:HMPC:c1} in which it is involved are performed sparsely by storing the matrix using the \textit{compressed sparse row} representation, this reduction can have a significant impact on the computation time of the algorithm, especially if $E$ and $F$ are dense and/or $n_y$ is large.

\begin{remark}
The key point of Algorithm \ref{alg:ADMM:HMPC} is that we are able to project directly onto the sets $\cD(\yUBj, \yLBj)$ thanks to the explicit solution provided by Theorem \ref{theo:proj:D}.
Otherwise, we would have had to consider the sets $\cKu(\yLBj)$ and $\cKl(\yUBj)$ separately for each $i \in \Z_1^{n_y}$, which would have increased the dimension of $s$ and therefore the computational complexity of the solver.
\end{remark}

\section{Numerical results} \label{sec:results}

This section shows two case studies evaluating the performance of the proposed HMPC solver.
The first one compares it against other solvers and MPC formulations from the literature.
The second one highlights the benefits obtained by considering the sets $\cD$ instead of SOC constraints.

\vspace*{-0.2em}
\subsection{Application of HMPC to a ball and plate system}

In this section we solve the problem from the case study of the original HMPC article \cite{Krupa_TAC_2022} using the results obtained in the previous sections as well as with the SCS solver \cite{ODonoghue_SCS_21}.
We also apply state-of-the-art solvers for the MPC for tracking (MPCT) formulation \cite{Ferramosca_A_2009} and for the standard MPC formulation with a terminal equality constraint (equMPC) described in \cite[Eq. (9)]{Krupa_TCST_20}, which are the MPC formulations used to highlight the benefits of HMPC in \cite{Krupa_TAC_2022} and \cite{Krupa_CDC_19}.

The system under consideration, described in detail in \cite[\S V.A]{Krupa_TAC_2022}, is a ball and plate system where the control objective is to steer the position of the ball to a given point by acting on the tilt of the plate through two motors on its main axes whose acceleration we can manipulate.

We maintain the same parameters, constraints and setup from \cite[\S V]{Krupa_TAC_2022}, including the prediction horizon $N = 5$ and base frequency $w = 0.3254$ (which is selected according to the criteria presented in \cite[\S VI]{Krupa_TAC_2022}) for the HMPC formulation, with two exceptions.
First, to improve the numerical conditioning of the resulting optimization problems, we scale the states corresponding to the position of the ball on each axis by a factor of $0.1$.
Second, due to the previous change, we reduced $T_h$ and $S_h$ by a factor of $10$ to maintain nearly indistinguishable closed-loop trajectories from the ones presented in \cite[\S V]{Krupa_TAC_2022}.
This resulted in an improvement on the number of iterations of the solvers by up to two orders of magnitude while maintaining the original results.

\begin{figure*}[t]
    \centering
    \begin{subfigure}[ht]{0.43\textwidth}
        \includegraphics[width=0.99\linewidth]{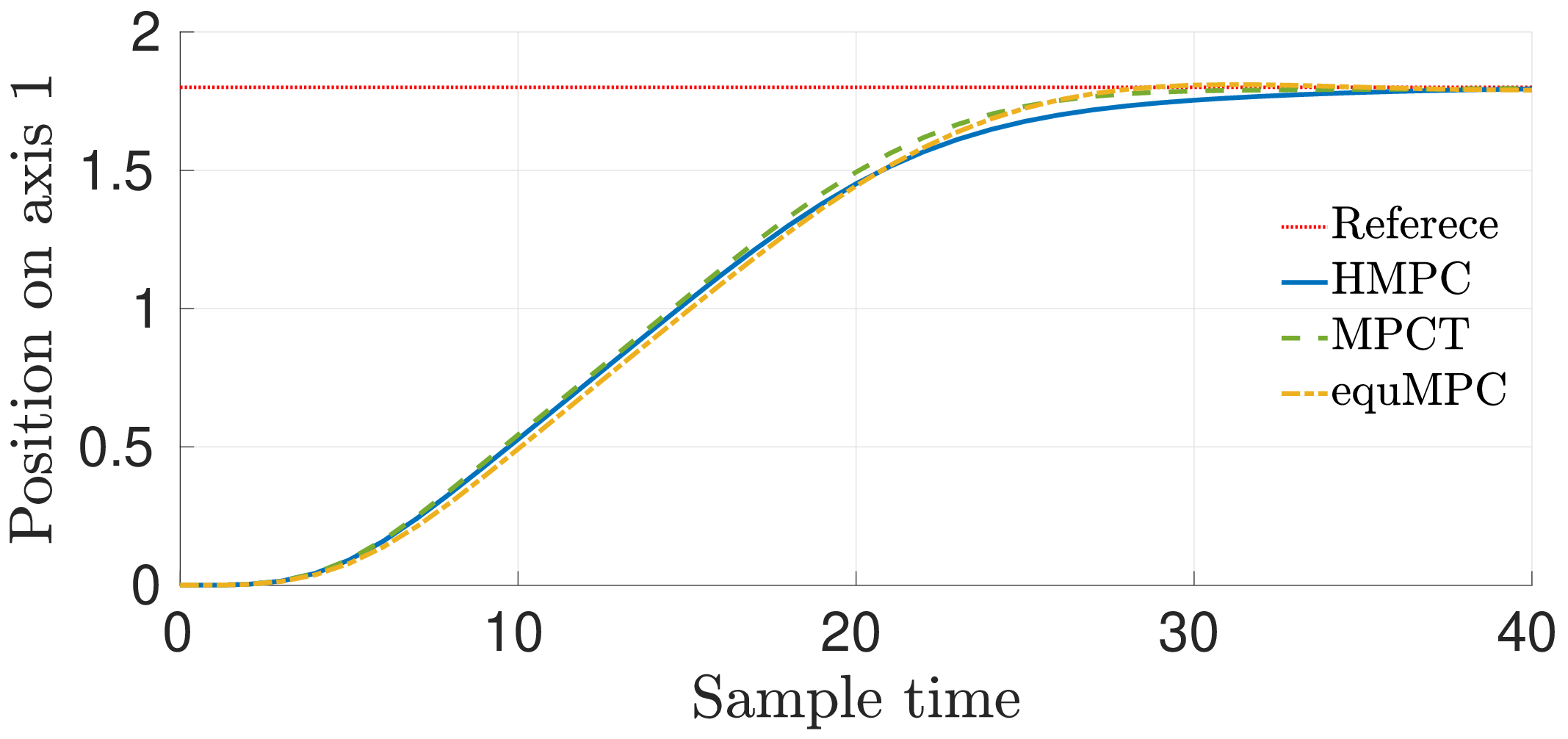}
        \caption{Position of ball on axis 1.}
        \label{fig:result:state}
    \end{subfigure}%
    \hfill%%
    \begin{subfigure}[ht]{0.43\textwidth}
        \includegraphics[width=0.99\linewidth]{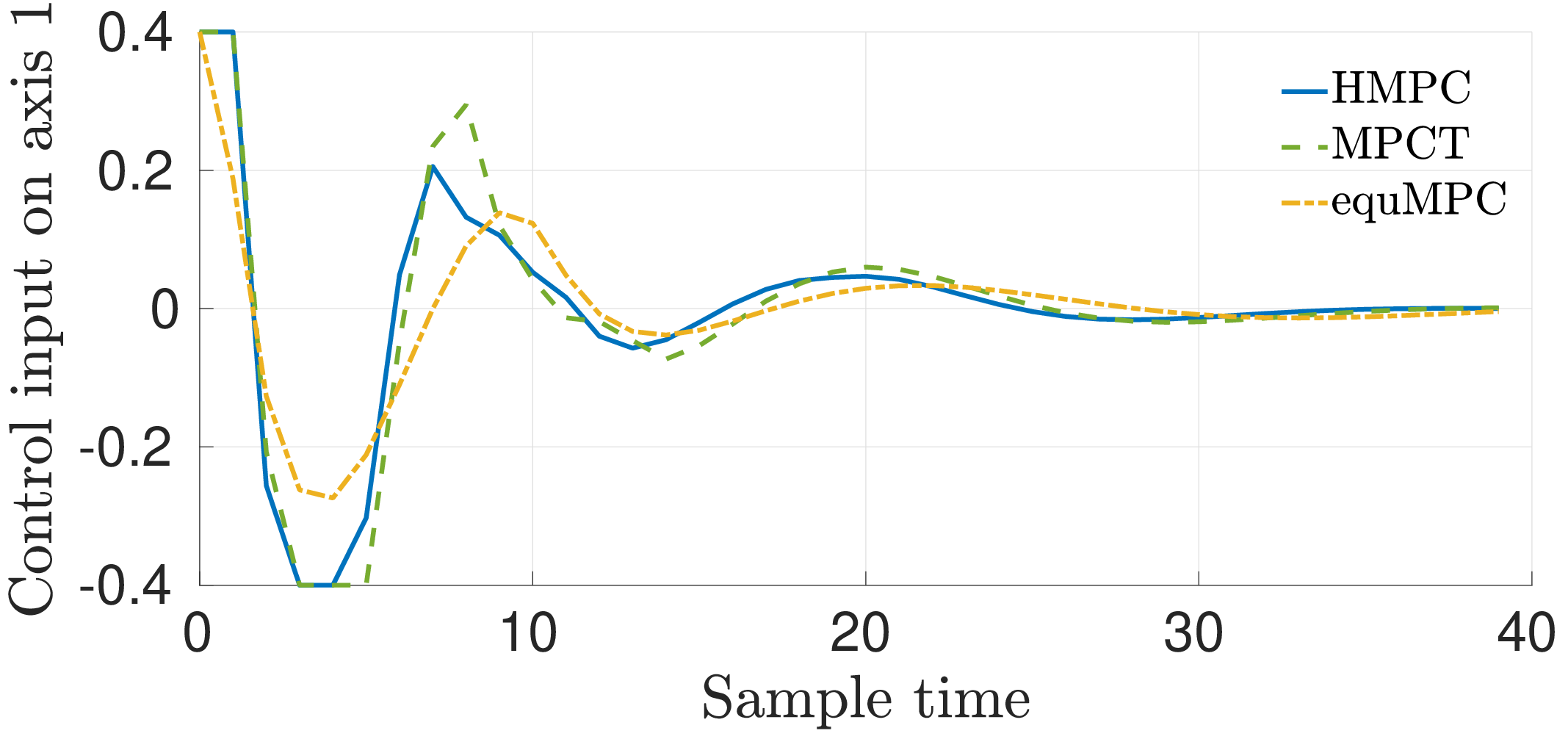}
        \caption{Control input for axis 1.}
        \label{fig:result:input}
    \end{subfigure}%
    \hfill%%

    \begin{subfigure}[ht]{0.43\textwidth}
        \includegraphics[width=0.99\linewidth]{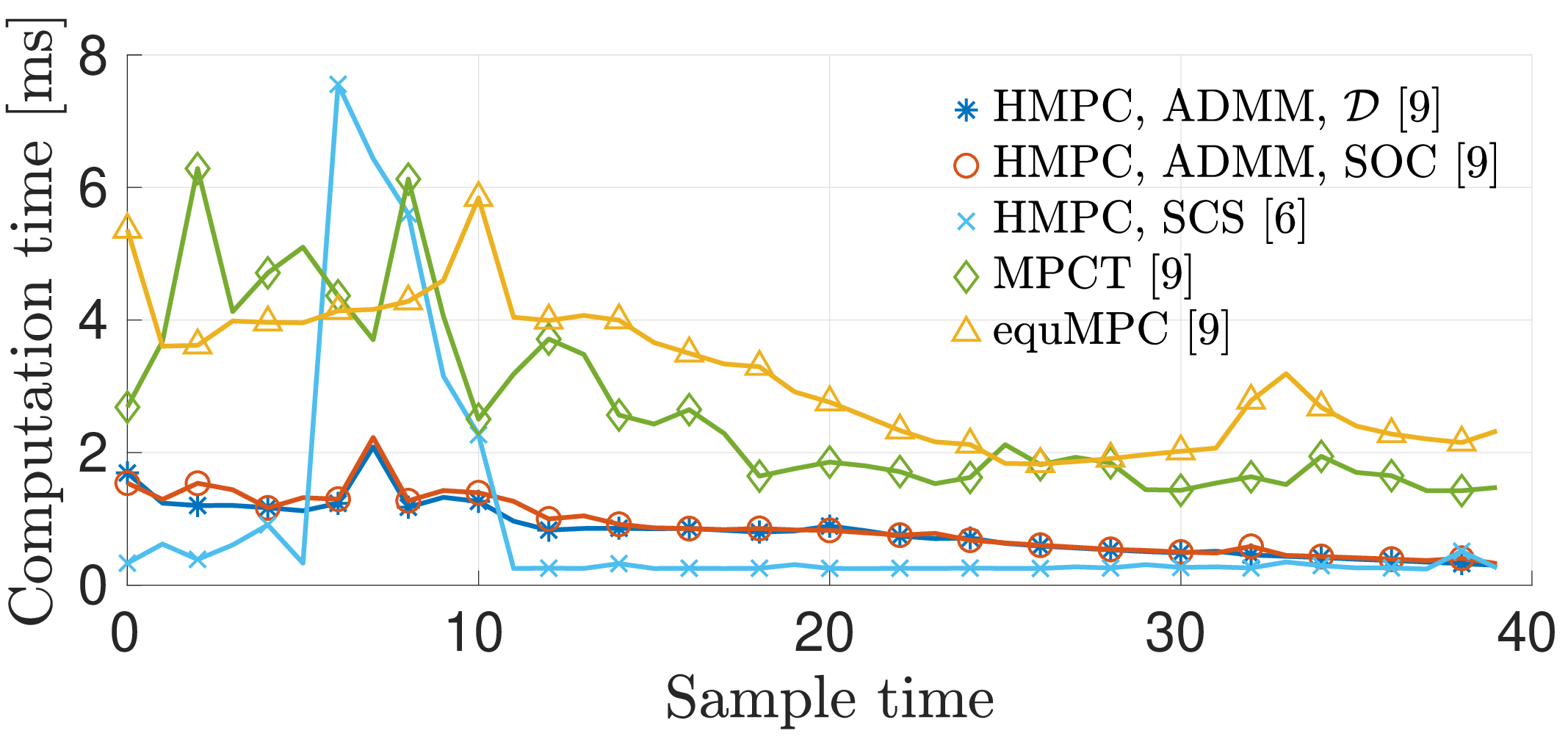}
        \caption{Computation times.}
        \label{fig:result:time}
    \end{subfigure}%
    \hfill%%
    \begin{subfigure}[ht]{0.43\textwidth}
        \includegraphics[width=0.99\linewidth]{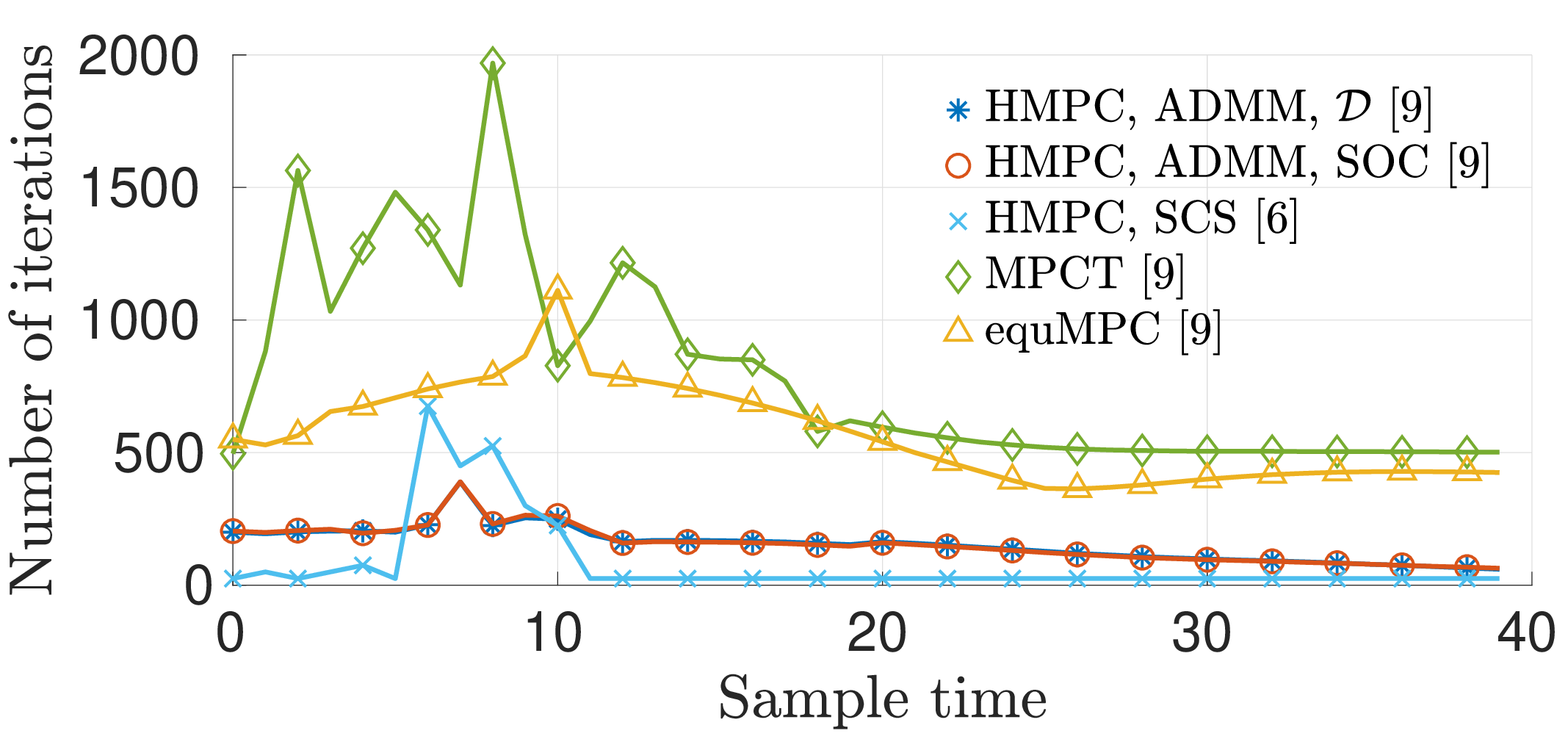}
        \caption{Number of iterations.}
        \label{fig:result:iter}
    \end{subfigure}%
    \caption{Closed-loop results on the ball and plate system.}
    \label{fig:result}
\end{figure*}

{\renewcommand{\arraystretch}{1.0}%
    \begin{table*}[t]
    \setlength{\tabcolsep}{3.5pt}
    \centering
	\begin{threeparttable}
    \begin{tabular}{rcccccccccc}
        & \multicolumn{4}{c}{Computation time [ms]} & \multicolumn{4}{c}{Number of iterations} & \\
        \cmidrule(lr){2-5}\cmidrule(lr){6-9}\cmidrule(lr){10-10}\cmidrule(lr){11-11}
        & Average & Median & Maximum & Minimum & Average & Median & Maximum & Minimum & Avrg. [$\mu$s/iters.] & Penalty parameters \\
        \cmidrule(lr){2-5}\cmidrule(lr){6-9}\cmidrule(lr){10-10}\cmidrule(lr){11-11}
        HMPC, ADMM, $\cD$   & 0.83 & 0.82 & 2.09 & 0.30 & 154.6 & 158.0 & 389  & 60  & 5.36 & $\rho = 15$ \\
        HMPC, ADMM, SOC     & 0.87 & 0.83 & 2.23 & 0.33 & 153.8 & 152.5 & 390  & 64  & 5.77 & $\rho = 15$ \\
        HMPC, SCS           & 0.91 & 0.26 & 7.56 & 0.24 & 78.8  & 25.0  & 675  & 25  & 11.5 & - \\
        MPCT                & 2.61 & 1.94 & 6.29 & 1.43 & 789.8 & 577.0 & 1969 & 496 & 3.31 & $\rho_1 = 3000$, $\rho_2 = 100$ \\
        equMPC              & 3.14 & 3.05 & 5.84 & 1.83 & 567.5 & 535.0 & 1112 & 363 & 5.54 & $\rho = 45$ \\
        \cmidrule(lr){2-11}
        % \bottomrule
    \end{tabular}
	\begin{tablenotes}[] \footnotesize
    \item \hspace{9em}Penalty parameter $\rho_1$ is for the constraints listed in \cite[Remark 4]{Krupa_TCST_21}, and $\rho_2$ for the rest.
	\end{tablenotes}
    \caption{Comparison between the different solvers during the simulation shown in Figure \ref{fig:result}.}
    \label{tab:result}
	\end{threeparttable}
\end{table*}}

We show results using the following solvers (formulations):
\begin{itemize}[leftmargin=*]
    \setlength\itemsep{0.0em}
\item SCS \texttt{v3.0.0} (HMPC) \cite{ODonoghue_SCS_21}. This state-of-the-art operator splitting solver can be applied to HMPC by imposing constraints \eqref{eq:HMPC:D} using SOC constraints. In particular, two SOC constraints are required for each constraints in \eqref{eq:HMPC:D}.
\item SPCIES \texttt{v0.3.7} (HMPC) \cite{SPCIES}. This solver implements the ADMM algorithm described in Section \ref{sec:solver}.
    We also show the results of a version in which we do not make use of the results presented in Section \ref{sec:SOC}, i.e., in which we impose constraints \eqref{eq:HMPC:D} using SOC constraints as in the SCS solver.
    Step \ref{alg:ADMM:HMPC:z} is solved using the dense method \eqref{eq:SADMM:solve:system:dense}, since it provided better computational results.
\item SPCIES \texttt{v0.3.7} (MPCT and equMPC) \cite{SPCIES}. The MPCT formulation is solved using the extended ADMM algorithm presented in \cite{Krupa_TCST_21}. The equMPC formulation is solved using the ADMM algorithm presented in \cite{Krupa_TCST_20} (see \cite[\S 5.4.2]{Krupa_Thesis_21} for a more in-depth explanation).
    The MPCT formulation uses the same parameters as the ones from \cite[\S V]{Krupa_TAC_2022}, including the prediction horizon $N = 15$, which is the smallest one for which the closed-loop performance of the system is similar to the one obtained using HMPC with $N = 5$.
    The prediction horizon of the equMPC formulation is taken as $N = 30$ for this very same reason.
    We take its $Q$ and $R$ ingredients as the ones used by the other two formulations.
\end{itemize}

The options of the solvers where left at their default values except for the scaling option of the SCS solver, which was set to $1$ since it significantly improved its performance, and the exit tolerances, which were set to $10^{-5}$.

Figure \ref{fig:result} shows the closed-loop results of the linear system with the three MPC formulations.
We refer the reader to \cite[\S V.B]{Krupa_TAC_2022} for some results on the application of HMPC to the nonlinear system.
Figure \ref{fig:result:state} shows the position of the ball on axis $1$, which converges to its reference $1.8$, and Figure \ref{fig:result:input} shows the control action on axis $1$, whose reference is $0$.
Figures \ref{fig:result:time} and \ref{fig:result:iter} show the computation times and number of iterations of each solver, with Table \ref{tab:result} showing relevant information about them.
It also shows the penalty parameters used for the solvers from the SPCIES toolbox.
The tests are performed on an Intel Core i5-8250U operating at $1.60$GHz in Matlab using the C-MEX interface of the solvers.

\subsection{Benefits of considering the sets $\cD$}

The results from the previous subsection seem to indicate that there is little benefit to using the sets $\cD$ over the alternative of the SOC constraints.
However, this is due to the small dimension and sparsity of matrices $E$ and $F$ in the previous example, which are simply imposing box constraints on four of the states and the two control inputs.

\begin{figure}[t]
    \centering
    \includegraphics[width=0.8\linewidth]{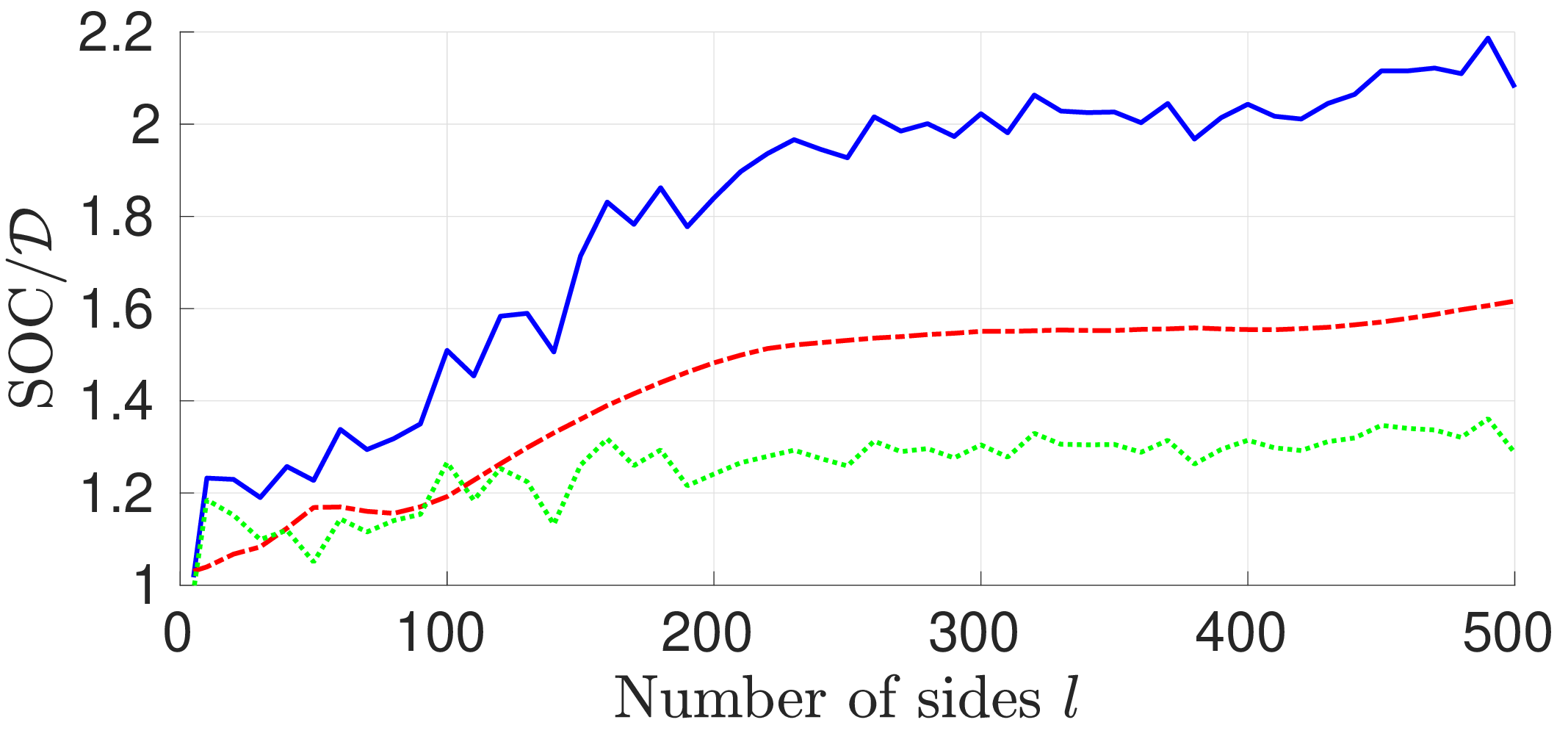}
    \caption{Comparison of using the $\cD$ or SOC constraints. Lines show averages when using the SOC constraints divided by averages when using the $\cD$ constraints for increasing values of $l$. Solid blue line represents the total computation time, dashed red the number of iterations and dotted green the computation time per iteration.}
    \label{fig:D:vs:SOC}
\end{figure}

However, for larger dimensions of $E$ and $F$ the advantage of using the sets $\cD$ becomes more pronounced.
To show this, we take the exact same setup as in the previous subsection, but we include constraints on the position of the ball on the plate in the form of a regular $l$-sided polygon centered at the origin whose vertices are at a distance of $2$ decimeters from the origin.
Note that each side of the polygon adds an additional row to the matrices $E$ and $F$ containing two non-zero elements, thus increasing $n_y$ by one.

Figure \ref{fig:D:vs:SOC} shows a comparison between the ADMM algorithm using the $\cD$ constraints and the SOC constraints for increasing values of $l$ starting from $l = 5$.
As can be seen, the average total computation time using the SOC constraints becomes over twice as long as using the $\cD$ constraints when $n_y$ becomes sufficiently large.

There are two factors at play that produce these results.
On one hand, increasing $n_y$ increases the dimension of $s$, and thus $C$.
When using the $\cD$ constraints this dimension is given by $N n_y + 3 n_y$, whereas using the SOC constraints the dimension is given by $N n_y + 6 n_y$.
For small values of $N$, where the HMPC formulation excels, this difference can become significant, becoming more prominent the larger the value of $n_y$ and the larger the number of non-zero elements in $E$ and $F$, since the cost of the multiplications by $C$ becomes the main computational burden of the algorithm in this scenario.
This reduction of the computational cost per iteration leads to the result presented by the dotted green line of Figure \ref{fig:D:vs:SOC}.

\begin{remark}
We note that the above discussion applies when the multiplications involving $C$ are performed without taking advantage of its structure, since if sets $\cD$ are not used, then the rows of $C$ related to the SOC constraints are (nearly) duplicated.
This can be used to reduce the computational burden of performing the operation $C z^{k+1}$ in Step \ref{alg:ADMM:HMPC:c1} of Algorithm \ref{alg:ADMM:HMPC}.
However, the same cannot be done in Step \ref{alg:ADMM:HMPC:q_hat}.
Thus, our approach still provides a computational reduction even if the particular structure of $C$ is exploited.
\end{remark}

The second factor that leads to an increase of the total computation time is the fact that the algorithm requires more iterations to satisfy the exit condition for larger values of $s$.
This is due to the fact that there are more primal and dual variables which must converge to a vicinity of their optimal values.
Thus, by using the sets $\cD$ we reduce these variables and thus the expected number of iterations of the algorithm.
This effect can be seen in the dashed red line of Figure \ref{fig:D:vs:SOC}.
Additionally, the use of sets $\cD$ may also lead to a reduction of the number of active constraints, which typically also reduces the number of ADMM iterations.

\begin{remark}
The use of sets $\cD$ is motivated by our consideration of the system constraints \eqref{eq:constraints}.
In systems that do not have upper and lower bounds, but instead only have one of the two, the proposed approach is meaningless, since this case would not result in pairs of opposed SOC constraints.
We note, however, that \eqref{eq:constraints} is a very common constraint in MPC and that our approach also applies to the case of box constraints on states and inputs.
\end{remark}

\section{Discussion and conclusions} \label{sec:conclusions}

This paper discusses how to efficiently solve the HMPC formulation by imposing its SOC-like constraints by using the sets $\cD$ defined in \eqref{eq:D}.
The main conclusion is that the HMPC formulation is suitable candidate for its implementation in embedded systems, since we obtain computational results that are in line (or even better) than the ones obtained for other linear MPC formulations.
In particular, we derive a few interesting conclusions from the numerical results:
\begin{itemize}[leftmargin=*]
    \setlength\itemsep{0.0em}
    \item A comparison between the solvers for HMPC and the ones for MPCT and equMPC shows that a solution of HMPC can be obtained in computation times comparable to the ones for MPC formulations whose optimization problem is a QP using state-of-the-art solvers when comparing prediction horizons for which the closed-loop performances are similar.
    \item A comparison between the ADMM solvers using the $\cD$ or SOC constraints shows that the use of the sets $\cD$ can have a significant impact on the computation times, especially for dense matrices $E$ and $F$ and sufficiently large values of $n_y$.
    For large prediction horizons, the reduction obtained on matrix $C$ when using sets $\cD$ may not be so noticeable, but in the case of HMPC, which is designed to be used with small prediction horizons, it can have a noticeable impact.
\item The results using the SCS solver indicate that lower computation times may be obtained if additional aspects, such as numerical preconditioning, adaptation of the penalty parameter, etc., were to be included in the algorithm.
\end{itemize}
We also note that the results of this paper could be very useful in other optimization settings involving constraints $\cD$.
An interesting future research line is to extend the solver to be able to deal with the complications that arise in a practical setting due to linearization errors or external disturbances.
In particular, the extension of HMPC to the robust case may be a viable and interesting topic for further future research.

% Fakesection Appendix
\begin{appendix}

\subsection{Proof of Theorem \ref{theo:proj:K}} \label{app:proof:proj:K}

This proof makes use of the following well known \textit{projection theorem} (see \cite[Prop. 1.1.9]{Bertsekas_Convex_2009}).

\begin{theorem}[Projection Theorem] \label{theorem:projection}
Let $\cc{C}$ be a nonempty closed convex subset of $\R^n$, and let $z$ be a vector in $\R^n$. There exists a unique vector that minimizes $\| v - z \|$ over $v \in \cc{C}$, called the projection of $z$ on $\cc{C}$. Furthermore, a vector $v^* \in \cc{C}$ is the projection of $z$ on $\cc{C}$ if and only if $\langle v - v^*, z - v^* \rangle \leq 0, \; \forall v \in \cc{C}$.
\end{theorem}

\begin{figure*}[!b]
\vspace*{-0.4em}
\rule[1ex]{\textwidth}{0.1pt}
\vspace*{-0.4em}
\small
\begin{equation*}
\vspace*{-0.4em}
\begin{aligned} \label{eq:proof:proj:K:final}
    \langle (y_0,\; y_1) - (\hat z_0 + c, \hat z_1), (z_0,\; z_1) - (\hat z_0 + c,\; \hat z_1) \rangle &= \langle (y_0 - \hat z_0 - c,\; y_1 - \alpha \hat z_0 \bz), (z_0 - \hat z_0 - c,\; (\|z_1\| - \alpha \hat z_0) \bz) \rangle \\
            &= \langle y_1 - \alpha \hat z_0 \bz,\; (\|z_1\| - \alpha \hat z_0) \bz \rangle + (y_0 - \hat z_0 - c)(z_0 - \hat z_0 - c) \\
            &= (\|z_1\| - \alpha \hat z_0) \langle y_1,\; \bz \rangle + \alpha \hat z_0 (\alpha \hat z_0 - \|z_1\|) \langle \bz,\; \bz \rangle + (y_0 - \hat z_0 - c)(z_0 - \hat z_0 - c) \\
            &\becauseof[\leq]{(*)} (\|z_1\| - \alpha \hat z_0) \|y_1\| \| \bz \| + \alpha \hat z_0 (\alpha \hat z_0 - \|z_1\|) \| \bz \|^2 + (y_0 - \hat z_0 - c)(z_0 - \hat z_0 - c) \\
            &\becauseof[\leq]{(*)} (\|z_1\| - \alpha \hat z_0) \alpha (y_0 - c) + \alpha \hat z_0 (\alpha \hat z_0 - \|z_1\|) + (y_0 - \hat z_0 - c)(z_0 - \hat z_0 - c) \\
            &= \alpha(\|z_1\| - \alpha \hat z_0) (y_0 - \hat z_0 - c) + (y_0 - \hat z_0 - c)(z_0 - \hat z_0 - c) \\
            &= (y_0 - \hat z_0 - c) \left( \alpha(\|z_1\| - \alpha \hat z_0) + z_0 - \hat z_0 - c \right) \\
            &= (y_0 - \hat z_0 - c) ( z_0 - c + \alpha \| z_1 \| - 2 \hat{z}_0 ) = (y_0 - \hat z_0 - c) (2 \alpha \tau - 2 \alpha \tau) = 0
\end{aligned}
\end{equation*}
\vspace*{-1.9em}
\normalsize
\end{figure*}

The first case is obvious: if $z = (z_0, z_1) \in \cK$, then $z = \cP_\cK(z)$.
To prove the second case, i.e., $\|z_1\| \leq -\alpha(z_0 - c)$, we note that $(c, 0) \in \cK$.
We now use Theorem \ref{theorem:projection} to prove that $(c, 0) = \cP_{\cK}(z)$. We have that, for any $y = (y_0, y_1) \in \cK$,
\small
\begin{equation*}
    \begin{aligned}
        &\langle (y_0, y_1) - (c, 0), (z_0, z_1) - (c, 0) \rangle\\
        & = \langle ( y_0 - c, y_1), (z_0 - c, z_1) \rangle = \langle y_1, z_1 \rangle + (y_0 - c)(z_0 - c)\\
        & \becauseof[\leq]{(*)} \|y_1\| \|z_1\| + (y_0 - c)(z_0 - c) \\
        & \moveEq{-4}\becauseof[\leq]{(**)} -\alpha^2(y_0 - c)(z_0 - c) + (y_0 - c)(z_0 - c) = 0, 
    \end{aligned}
\end{equation*}
\normalsize
where $(*)$ is due to the Cauchy-Schwarz inequality and $(**)$ holds because $\| z_1 || \leq -\alpha (z_0 - c)$ and $\| y_1 \| \leq \alpha (y_0 - c)$.
Next, we prove the third case.
Let us introduce the notation
\begin{equation*}
\hat z_0 \doteq \alpha \tau, \quad \bz \doteq \fracg{z_1}{\|z_1\|}, \quad \hat z_1 \doteq \alpha \hat z_0 \bz,
\end{equation*}
where we recall that $\tau = 0.5 (\alpha(z_0 - c) + \| z_1 \|)$ and note that $\bz$ is well defined because $\| z_1 \| \not = 0$ due to the non-satisfaction of the conditions in \eqref{eq:proj:K:1} and \eqref{eq:proj:K:2}.
We start by proving that $(\hat z_0 + c, \hat z_1) \in \cK$ when \textit{(i)} $\|z_1\| > \alpha(z_0 - c)$ and \textit{(ii)} $\|z_1\| > -\alpha(z_0 - c)$, i.e., that $\| \hat z_1 \| \leq \alpha (\hat z_0 + c - c) = \alpha \hat z_0$.
Indeed, 
\begin{equation*}
    \begin{aligned}
        \| \hat z_1 \| = \big\| \tau \fracg{z_1}{\|z_1\|} \big\| = | \tau | \becauseof{\textit{(ii)}} \tau, \quad
            \alpha \hat z_0 = \tau.
    \end{aligned}
\end{equation*}
Finally, we show that $\cP_\cK(z) = (\hat z_0 + c, \hat z_1)$.
Pick an arbitrary $y = (y_0, y_1) \in \cK$ and fix $z = (z_0, z_1)$ satisfying \textit{(i)} and \textit{(ii)}.
The proof follows from the equation displayed at the bottom of the next page, where the steps marked with $(*)$ hold because $\| z_1 \| - \alpha \hat{z}_0 = \fracg{1}{2} (\| z_1 \| - \alpha(z_0 - c) ) \becauseof[>]{\textit{(i)}} 0$. \qed

\subsection{Proof of Theorem \ref{theo:proj:D}} \label{app:proof:proj:D}

Since we assume that $\zLB \leq \zUB$, we have from Lemma \ref{lemma:nonempty:D} that $\cD$ is non-empty.
Thus, the projection of $z$ onto $\cD$ exists and is unique, since it is a closed convex set.
Moreover, since it is the intersection of two closed convex sets, the iterates of Algorithm \ref{alg:Dykstra} will converge to the projection of $z$ onto $\cD$ \cite[Theorem 2]{Boyle_1986}.
We will now show that, in fact, Algorithm \ref{alg:Dykstra} will converge after a single iteration.
That is, taking first the projection onto $\cKu$ and then onto $\cKl$, we show that $v^2 = w^1$, which along with Lemma \ref{lemma:convergence:Dykstra} proves the claim.

Let us denote $v^k = (v_0^k, v_1^k) \in \R \times \R^{n-1}$ and $w^k = (w_0^k, w_1^k) \in \R \times \R^{n-1}$.
We divide the proof into several cases.
    
\noindent \textbf{\textit{Case 1:}} $z \in \cKl$ and $z \in \cKu$. This case is trivial, since $z \in \cD$.

\noindent \textbf{\textit{Case 2:}} $z \in \cKu$ and $\cP_{\cKl}(v^1)$ is obtained from \eqref{eq:proj:K:2}. Then, $v^1 = z$, $p^1 = 0$, $w^1 = (\zUB, 0)$. Since $\zUB \geq \zLB$, it is easy to see that $(\zUB, 0) \in \cKu$. Therefore, $v^2 = \cP_{\cKu}((\zUB, 0) + p^1) = (\zUB, 0) = w^1$.

\noindent \textbf{\textit{Case 3:}} $z \in \cKu$ and $\cP_{\cKl}(v^1)$ is obtained from \eqref{eq:proj:K:3}. Then, $v^1 = z$, $p^1 = 0$,
\begin{equation*}
w^1 = \frac{1}{2}( \zUB - z_0 + \|z_1\|) \left( -1, \fracg{z_1}{\|z_1\|} \right) + \left( \zUB, 0 \right),
\end{equation*}
and $v^2 = \cP_{\cKu}(w^1 + p^1) = \cP_{\cKu}(w^1)$. All that remains is to show that $w^1 \in \cKu$, which we prove by contradiction.
Assume that $w^1 \not\in \cKu$, i.e., $\|w^1_1\| > w^1_0 - \zLB$, which can be expressed~as 
\begin{equation} \label{eq:case:3}
\left| \frac{1}{2}(\zUB - z_0 + \|z_1\|) \right| > \frac{1}{2} (z_0 - \zUB - \|z_1\|) + \zUB - \zLB.
\end{equation}
If $\frac{1}{2}(\zUB - z_0 + \|z_1\|) > 0$, then \eqref{eq:case:3} leads to the contradiction $|| z_1 \| > z_0 - \zLB$, since we assume that $z \in \cKu$.
If it is smaller of equal to $0$, then \eqref{eq:case:3} leads to the contradiction $\zLB > \zUB$. 

\noindent \textbf{\textit{Case 4:}} $\cP_{\cKu}(z)$ is obtained from \eqref{eq:proj:K:2}. Then, $v^1 = (\zLB, 0)$, $p^1 = z - (\zLB, 0)$. Since $\zLB \leq \zUB$, it is easy to see that $(\zLB, 0) \in \cKl$. Therefore, $w^1 = \cP_{\cKl}(v^1) = (\zLB, 0)$, which leads to
\begin{align*}
    v^2 &= \cP_{\cKu}(w^1 + p^1) = \cP_{\cKu} ( (\zLB, 0) + z - (\zLB, 0)) \\
        &= \cP_{\cKu}(z) = (\zLB, 0) = w^1.
\end{align*}

\noindent \textbf{\textit{Case 5:}} $\cP_{\cKu}(z)$ is obtained from \eqref{eq:proj:K:3}. Then, 
\begin{equation} \label{eq:case:5:z-}
v^1 = \frac{1}{2}( z_0 - \zLB + \|z_1\|) \left( 1, \fracg{z_1}{\|z_1\|} \right) + \left( \zLB, 0 \right),
\end{equation}
and $p^1 = z - v^1$. We now distinguish between two subcases.
\begin{enumerate}[label=(\roman*)]
    \item $v^1 \in \cKl$. Then, $w^1 = v^1$, which leads to, $v^2 = \cP_{\cKu}(v^1 + z - v^1) = \cP_{\cKu}(z) = v^1 = w^1$.
     \item $v^1 \notin \cKl$. From \eqref{eq:case:5:z-}, the use of simple algebra leads to $\|v^1_1\| = v^1_0 - \zLB$.
         If $\zUB > \zLB$ this implies that $\|v^1_1\| > v^1_0 - \zUB$, and thus $w^1$ is computed using \eqref{eq:proj:K:3}, i.e.,
    \begin{align*}
        w^1 &= \frac{1}{2} (\zUB - v^1_0 + v^1_0 - \zLB) \left( -1, \fracg{v^1_1}{\|v^1_1\|} \right) + (\zUB, 0) \\
            &= \frac{1}{2} \left( \zUB + \zLB, \fracg{\zUB - \zLB}{\|v^1_1\|} v^1_1 \right),
    \end{align*}
    which along the use of simple algebra leads to $\|w^1_1\| = w^1_0 - \zLB$.
    If, on the other hand, $\zUB = \zLB$, then $w^1$ is computed using \eqref{eq:proj:K:2}, i.e., $w^1 = (\zUB, 0) = (\zLB, 0)$.
    In both cases it is easy to see that $w^1 \in \cKu$ and thus $v^2 = \cP_{\cKu}(w^1) = w^1$.
\end{enumerate}
These cases cover all the possibilities for projecting first onto $\cKu$ and then onto $\cKl$. \qed

\end{appendix}

% Fakesection Bibliography
\bibliographystyle{IEEEtran}
\bibliography{IEEEabrv,BibKrupa}

\begin{thebibliography}{10}
\providecommand{\url}[1]{#1}
\csname url@samestyle\endcsname
\providecommand{\newblock}{\relax}
\providecommand{\bibinfo}[2]{#2}
\providecommand{\BIBentrySTDinterwordspacing}{\spaceskip=0pt\relax}
\providecommand{\BIBentryALTinterwordstretchfactor}{4}
\providecommand{\BIBentryALTinterwordspacing}{\spaceskip=\fontdimen2\font plus
\BIBentryALTinterwordstretchfactor\fontdimen3\font minus
  \fontdimen4\font\relax}
\providecommand{\BIBforeignlanguage}[2]{{%
\expandafter\ifx\csname l@#1\endcsname\relax
\typeout{** WARNING: IEEEtran.bst: No hyphenation pattern has been}%
\typeout{** loaded for the language `#1'. Using the pattern for}%
\typeout{** the default language instead.}%
\else
\language=\csname l@#1\endcsname
\fi
#2}}
\providecommand{\BIBdecl}{\relax}
\BIBdecl

\bibitem{Krupa_TAC_2022}
P.~Krupa, D.~Limon, and T.~Alamo, ``Harmonic based model predictive control for
  set-point tracking,'' \emph{IEEE Transactions on Automatic Control}, vol.~67,
  no.~1, pp. 48--62, 2022.

\bibitem{Krupa_CDC_19}
P.~Krupa, M.~Pereira, D.~Limon, and T.~Alamo, ``Single harmonic based model
  predictive control for tracking,'' in \emph{58th Conference on Decision and
  Control (CDC)}.\hskip 1em plus 0.5em minus 0.4em\relax IEEE, 2019, pp.
  151--156.

\bibitem{Camacho_S_2013}
E.~F. Camacho and C.~B. Alba, \emph{{M}odel {P}redictive {C}ontrol},
  2nd~ed.\hskip 1em plus 0.5em minus 0.4em\relax London, UK: Springer-Verlag,
  2007.

\bibitem{Garstka_JOTA_2021}
M.~Garstka, M.~Cannon, and P.~Goulart, ``{COSMO}: A conic operator splitting
  method for convex conic problems,'' \emph{Journal of Optimization Theory and
  Applications}, vol. 190, no.~3, pp. 779--810, 2021.

\bibitem{Stellato_OSQP}
B.~Stellato, G.~Banjac, P.~Goulart, A.~Bemporad, and S.~Boyd, ``{OSQP}: An
  operator splitting solver for quadratic programs,'' \emph{Mathematical
  Programming Computation}, vol.~12, no.~4, pp. 637--672, 2020.

\bibitem{ODonoghue_SCS_21}
B.~O'Donoghue, ``Operator splitting for a homogeneous embedding of the linear
  complementarity problem,'' \emph{{SIAM} Journal on Optimization}, vol.~31,
  pp. 1999--2023, August 2021.

\bibitem{Boyd_FTML_2011}
S.~Boyd, N.~Parikh, E.~Chu, B.~Peleato, and J.~Eckstein, ``Distributed
  optimization and statistical learning via the alternating direction method of
  multipliers,'' \emph{Foundations and Trends in Machine Learning}, vol.~3,
  no.~1, pp. 1--122, 2011.

\bibitem{Krupa_TCST_20}
P.~Krupa, D.~Limon, and T.~Alamo, ``Implementation of model predictive control
  in programmable logic controllers,'' \emph{IEEE Transactions on Control
  Systems Technology}, vol.~29, no.~3, pp. 1117--1130, 2021.

\bibitem{SPCIES}
------, ``{SPCIES: Suite of Predictive Controllers for Industrial Embedded
  Systems},'' \url{https://github.com/GepocUS/Spcies}, Dec 2020.

\bibitem{Bauschke_Thesis_1996}
H.~H. Bauschke, ``Projection algorithms and monotone operators,'' Ph.D.
  dissertation, Theses (Dept. of Mathematics and Statistics)/Simon Fraser
  University, 1996.

\bibitem{Bertsekas_Convex_2009}
D.~P. Bertsekas, \emph{Convex {O}ptimization {T}heory}.\hskip 1em plus 0.5em
  minus 0.4em\relax Athena Scientific Belmont, 2009.

\bibitem{Boyle_1986}
J.~P. Boyle and R.~L. Dykstra, ``A method for finding projections onto the
  intersection of convex sets in {H}ilbert spaces,'' in \emph{Advances in order
  restricted statistical inference}.\hskip 1em plus 0.5em minus 0.4em\relax
  Springer, 1986, pp. 28--47.

\bibitem{Stovsic_projection_2016}
M.~Sto{\v{s}}i{\'c}, J.~Xavier, and M.~Dodig, ``Projection on the intersection
  of convex sets,'' \emph{Linear Algebra and its Applications}, vol. 509, pp.
  191--205, 2016.

\bibitem{Bauschke_Dykstra_94}
H.~H. Bauschke and J.~M. Borwein, ``Dykstra's alternating projection algorithm
  for two sets,'' \emph{Journal of Approximation Theory}, vol.~79, no.~3, pp.
  418--443, 1994.

\bibitem{Ferramosca_A_2009}
A.~Ferramosca, D.~Limon, I.~Alvarado, T.~Alamo, and E.~Camacho, ``{MPC} for
  tracking with optimal closed-loop performance,'' \emph{Automatica}, vol.~45,
  no.~8, pp. 1975--1978, 2009.

\bibitem{Limon_A_2008}
D.~Limon, I.~Alvarado, T.~Alamo, and E.~F. Camacho, ``{MPC} for tracking
  piecewise constant references for constrained linear systems,''
  \emph{Automatica}, vol.~44, no.~9, pp. 2382--2387, 2008.

\bibitem{Benzi_2005}
M.~Benzi, G.~H. Golub, and J.~Liesen, ``Numerical solution of saddle point
  problems,'' \emph{Acta numerica}, vol.~14, pp. 1--137, 2005.

\bibitem{Saraf_TAC_2019}
N.~Saraf and A.~Bemporad, ``A bounded-variable least-squares solver based on
  stable {QR} updates,'' \emph{IEEE Transactions on Automatic Control},
  vol.~65, no.~3, pp. 1242--1247, 2019.

\bibitem{Boyd_ConvexOptimization}
S.~Boyd, \emph{Convex Optimization}, 7th~ed.\hskip 1em plus 0.5em minus
  0.4em\relax Cambridge, UK: Cambridge University Press, 2009.

\bibitem{Krupa_TCST_21}
P.~Krupa, I.~Alvarado, D.~Limon, and T.~Alamo, ``Implementation of model
  predictive control for tracking in embedded systems using a sparse extended
  {ADMM} algorithm,'' \emph{Transactions on Control Systems Technology},
  vol.~30, no.~4, pp. 1798--1805, 2022.

\bibitem{Krupa_Thesis_21}
P.~Krupa, ``Implementation of {MPC} in embedded systems using first order
  methods,'' Ph.D. dissertation, University of Seville, 2021, available at
  {arXiv:2109.02140}.

\end{thebibliography}

\end{document}